\documentclass[11pt, a4paper, leqno]{amsart}
\usepackage{a4wide,amsfonts, amssymb, amsmath,amsthm}
\usepackage{mathrsfs}
\usepackage[abbrev]{amsrefs}

\newtheorem{theorem}{Theorem}[section]
\newtheorem{lemma}[theorem]{Lemma}

\newtheorem{proposition}[theorem]{Proposition}
\newtheorem{corollary}[theorem]{Corollary}

\theoremstyle{definition}
\newtheorem{remark}[theorem]{Remark}

\newcommand{\R}{{\mathbb R}}

\title{Nonlinear Inequalities with Double Riesz Potentials}

\author{Marius Ghergu}
\address{School of Mathematics and Statistics, University College Dublin, Belfield, Dublin 4, Ireland}
\address{Institute of Mathematics Simion Stoilow of the Romanian Academy, 21 Calea Grivitei St., 010702 Bucharest, Romania}
\email{marius.ghergu@ucd.ie}

\author{Zeng Liu}
\thanks{ZL was supported by NSFC Grant Numbers 11901418, 11771319.}
\address{Department of Mathematics, Suzhou University of Science and Technology, Suzhou, 215009, P.R. China}
\email{zliu@mail.usts.edu.cn}

\author{Yasuhito Miyamoto}
\thanks{YM was supported by JSPS KAKENHI Grant Numbers 19H01797, 19H05599.}
\address{Graduate School of Mathematical Sciences, The University of Tokyo, 3-8-1 Komaba, Meguro-ku, Tokyo 153-8914, Japan}
\email{miyamoto@ms.u-tokyo.ac.jp}

\author{Vitaly Moroz}
\address{Mathematics Department, Swansea University,
		Bay Campus, Fabian Way, Swansea SA1 8EN, Wales, United Kingdom}
\email{v.moroz@swansea.ac.uk}

\date{\today}

\begin{document}

\subjclass[2020]{Primary 45G10; Secondary 31B10, 45M05}

\date{}

\keywords{Nonlinear integral inequalities; Riesz potentials;
nonlocal positivity principle; Liouville theorems}

\begin{abstract}
We investigate the nonnegative solutions to the nonlinear integral inequality
$u \ge I_{\alpha}\ast\big((I_\beta \ast u^p)u^q\big)$ a.e. in $\R^N$,
where $\alpha, \beta\in (0,N)$, $p, q>0$ and $I_\alpha,I_\beta$ denote the Riesz potentials of  order $\alpha$ and $\beta$ respectively. Our approach relies on a nonlocal positivity principle which allows us to derive optimal ranges for the parameters $\alpha$, $\beta$, $p$ and $q$ to describe the existence and the nonexistence of a solution. The optimal decay at infinity for such solutions is also discussed.
\end{abstract}

\maketitle

	\section{Introduction}

	We study nonnegative solutions of the following integral inequality with double Riesz potentials
	\begin{equation}\label{eqC}
		u \ge I_{\alpha}\ast\big((I_\beta \ast u^p)u^q\big)\quad\text{a.e.~in}\quad\R^N,
	\end{equation}
in the range $N\ge 2$, $p,q>0$, $\alpha,\beta\in(0,N)$ and $I_\gamma(x):=A_\gamma|x|^{-(N-\gamma)}$ is the Riesz potential of order $\gamma\in(0,N)$ where $*$ denotes the standard convolution in $\R^N$. The choice of the normalisation constant $A_\gamma:=\frac{\Gamma((N-\gamma)/2)}{\pi^{N/2}2^{\gamma}\Gamma(\gamma/2)}$ 	ensures that $I_\gamma(x)$ could be interpreted as the Green function of $(-\Delta)^{\gamma/2}$ in $\R^N$, and that the semigroup property $I_{\alpha+\beta}=I_\alpha*I_\beta$ holds for all $\alpha,\beta\in(0,N)$ such that $\alpha+\beta<N$, see for example \cite{Landkof}*{p.\thinspace{}45}.

By a nonnegative {\em solution} of \eqref{eqC} we understand a function $u\in L^1_{loc}(\R^N)$, $u\geq 0$, such that the right hand side of \eqref{eqC} is well-defined, i.e.
 \begin{equation}\label{eqCdom}
 I_{\alpha}\ast\big((I_\beta \ast u^p)u^q\big)<+\infty\quad\text{a.e.~in}\quad\R^N
 \end{equation}
 and the inequality \eqref{eqC} holds a.e.~in $\R^N$. Condition \eqref{eqCdom} above is equivalent (see Lemma \ref{leq} below) to
\begin{equation}\label{cd}
(I_\beta \ast u^p)u^q \in L^1\big((1+|x|)^{-(N-\alpha)}dx,\R^N\big).
\end{equation} 	

Integral inequalities and equations featuring  a single Riesz potential have been extensively investigated in the past two decades. The prototype model
$$
u= I_\alpha\ast u^p \quad\text{ in }\quad\R^N,
$$
has been largely investigated starting with the seminal works in \cite{CLO2005,CLO2006}. More recently, various techniques have been devised to deal with systems of equations or inequalities that incorporate anisotropic or more general potentials \cite{CDM2008, GT2015, Lei2013, V2015, V2016}.

Our aim in this paper is to provide an optimal description for the existence and nonexistence of positive solutions to the integral inequality \eqref{eqC}.

For any $\alpha>0$, the fractional Laplacian $(-\Delta)^{\alpha/2}$ is defined by means of the Fourier transform
$$\widehat{(-\Delta)^{\alpha/2} u}(\xi):=|\xi|^\alpha\hat{u}(\xi)$$
for all $u\in\mathcal{S}'$ such that $|\xi|^\alpha u\in\mathcal{S}'$, here
$\mathcal{S}'$ stands for the space of tempered distributions on $\R^N$ which is the dual of the Schwartz space $\mathcal{S}$.

Since for $\alpha\in(0,N)$ the Riesz potential $I_{\alpha}$ can be interpreted as the inverse of $(-\Delta)^{\alpha/2}$  (cf. \cite{Stein}*{Sect.5.1} or \cite[Section 2.1]{CDM2008}), under some extra integrability conditions on $u\ge 0$, inequality \eqref{eqC} is equivalent to the elliptic inequality
	\begin{equation}\label{eqm}
		(-\Delta)^{\alpha/2} u \ge (I_\beta \ast u^p)u^q\quad\text{a.e. in $\R^N$},
	\end{equation}
	provided that both \eqref{eqC} and \eqref{eqm} are well-defined.
	This is the case, for instance, if  \eqref{cd} holds and $u$ belongs to the
	 homogeneous Sobolev space $\dot{H}^{\alpha/2}(\R^N)$, so that \eqref{eqm} is understood in the weak sense. Pointwise interpretations of the inequality \eqref{eqm} for non-integer $\alpha/2$ are also possible, cf. \cite[Theorem 2.13]{CDM2008}. For a comparison of different definitions of the higher order fractional Laplacian $(-\Delta)^{\alpha/2}$ see \cite{AJS}.

Inequality \eqref{eqm} is a Choquard type inequality. Equations and inequalities of such structure originate from mathematical physics and have attracted considerable interest of mathematicians in the past decade, see \cite{MVS-survey} for a survey. In the 2nd order elliptic case $\alpha=2$ optimal regimes for the existence and nonexistence of positive solutions to inequality \eqref{eqm} were fully investigated in \cite{MVS-JDE}.
The higher--order polyharmonic case $\alpha/2=m\in\mathbb N$ was recently studied in \cite{GMM}, where (amongst other results) optimal existence and nonexistence regimes for the equation \eqref{eqm} were obtained for the exponents $p\ge 1$ and $q>1$, see \cite[Theorem 1.4]{GMM}.

	In this work we extend earlier results in \cite{MVS-JDE} and \cite{GMM} to the full admissible range $\alpha\in(0,N)$ and exponents $p,q>0$. Our approach is different from the techniques in \cite{GMM}, which were based on the poly--superharmonic properties of $(-\Delta)^m$ in the elliptic framework of equation \eqref{eqm}. Instead, we work entirely with the double--nonlocal inequality \eqref{eqC}. Our analysis of \eqref{eqC} employs only direct Riesz kernel estimates, and a new version of the nonlocal positivity principle in Lemma \ref{lV}, inspired by \cite[Proposition 3.2]{MVS-JDE}. This has the advantage of incorporating the fractional case of noninteger $\alpha/2$ in a seemingly effortless way, and does not rely on comparison principles or Harnack type inequalities, which are commonly used for similar Liouville type results in the elliptic framework, but which are generally speaking not available in the case of the higher--order fractional Laplacians $(-\Delta)^{\alpha/2}$ with $\alpha>2$.
 \smallskip

The main result of this work related to the existence of positive solutions to \eqref{eqC} reads as follows.

 \begin{theorem}\label{t1}
 	Let $p,q>0$. Then, inequality \eqref{eqC} has a nontrivial nonnegative solution $u\in L^1_{loc}(\R^N)$ which satisfies \eqref{eqCdom} if and only if
$$
\begin{cases}
\displaystyle p>\frac{\beta}{N-\alpha}&\\[0.1in]
\displaystyle  p+q>\frac{N+\beta}{N-\alpha} &\\[0.1in]
\displaystyle  q>\frac{\beta}{N-\alpha}&\quad\mbox{ if }\beta>N-\alpha,\\[0.1in]
\displaystyle  q\geq 1 & \quad \mbox{ if }\beta=N-\alpha,\\[0.1in]
\displaystyle  q>1-\frac{N-\alpha-\beta}{N}p&\quad\mbox{ if }\beta<N-\alpha.
\end{cases} 	
$$
 \end{theorem}

The necessary part of the proof follows directly from Propositions \ref{p1},  \ref{p5}-\ref{p9} below.
The sufficiency follows from Propositions \ref{sharp-B1}-\ref{sharp-B1plusplus}, 
where we construct explicitly smooth positive radial solutions to \eqref{eqC}.
In the case $\alpha=2$ our results are fully consistent with the results established in \cite[Theorem 1]{MVS-JDE} for the 2nd order elliptic inequality \eqref{eqm}.
The nonexistence of positive solutions to double--nonlocal inequality \eqref{eqC} with $p>1$, $q>0$ and $p+q\le\frac{N+\beta}{N-\alpha}$ was established by different methods in \cite[Theorem 1]{Le}.

\begin{remark}\label{remark01}
In Section \ref{odec} we also discuss the optimal decay of solutions to \eqref{eqC} in terms of the parameters $\alpha$, $\beta$, $p$ and $q$.
Clearly (see \eqref{Rlow}), if $u\ge 0$ is a nontrivial solution of \eqref{eqC}  then
$\liminf_{|x| \to \infty} u (x) |x|^{N - \alpha}>0$. In particular, for $R\gg 1$ we have
\begin{equation}\label{2lowint+}
	\int_{B_{2R} \setminus B_R}u\ge cR^{\alpha}.
\end{equation}
In Proposition \ref{p10} we establish an integral lower bound
\begin{equation}\label{qlowint+}
	\int_{B_{2R} \setminus B_R}u\ge cR^{\frac{\alpha+\beta-Nq}{1-q}},
\end{equation}
which is stronger than \eqref{2lowint+} when  $\alpha+\beta<N$ and $q < \frac{\beta}{N - \alpha}<1$.
In Propositions \ref{sharp-B1} and \ref{sharp-B2} we construct positive radial solutions $u$ to \eqref{eqC} that confirm the optimality of \eqref{2lowint+} when  $q > \frac{\beta}{N - \alpha}$ and of \eqref{qlowint+} when  $q<\frac{\beta}{N-\alpha}$. When $q=\frac{\beta}{N-\alpha}$ the bounds in \eqref{2lowint+} and \eqref{qlowint+} coincide. In that case in Propositions \ref{sharp-B1plus} and \ref{sharp-B1plusplus} we construct positive radial solutions to \eqref{eqC} that satisfy \eqref{2lowint+} up to a $\log$ ($q<1$) or arbitrary small polynomial ($q=1$) corrections.
In the case $\alpha=2$ such corrections are necessary, see \cite[Proposition 4.12, 4.13]{MVS-JDE}.
\end{remark}

\section{Preliminaries}

In this section we collect some useful facts for our approach.

\begin{lemma}\label{leq}
Let  $f:\R^N\to\R$ be a nonnegative measurable function. Then, the Riesz potential $I_\alpha*f$ of order $\alpha\in(0,N)$ is well defined, in the sense that
\begin{equation}\label{Rwell}
I_\alpha*f<+\infty \quad \text{a.e. in $\R^N$,}
\end{equation}
if and only if
\begin{equation}\label{Rup}
f\in L^1\big((1+|x|)^{-(N-\alpha)}dx,\R^N\big).
\end{equation}
\end{lemma}
Moreover, if \eqref{Rwell} fails then $I_\alpha*f=+\infty$ everywhere in $\R^N$, see \cite{Landkof}*{p.61-62}.
We present the proof of the lemma for completeness.

\begin{proof} Assume first that \eqref{Rwell} holds. Then, for any $x,y\in \R^N$, $x\neq 0$ we have
$$
|x-y|^{N-\alpha}\leq c(|x|^{N-\alpha}+|y|^{N-\alpha})\leq c\max\{1, |x|^{N-\alpha}\} (1+|y|^{N-\alpha})
\leq c\max\{1, |x|^{N-\alpha}\} (1+|y|)^{N-\alpha}.
$$
Thus, for any $x\in \R^N\setminus\{0\}$ such that \eqref{Rwell} holds, we have
$$
\infty>(I_\alpha\ast f)(x)\geq \frac{1}{c\max\{1, |x|^{N-\alpha}\} }\int_{\R^N}\frac{f(y) dy}{(1+|y|)^{N-\alpha}},
$$
which yields \eqref{Rup}.

Conversely, assume now that \eqref{Rup} holds and let $x\in \R^N\setminus\{0\}$. From \eqref{Rup} we have $f\in L^1_{loc}(\R^N)$. Then
$$
\begin{aligned}
(I_\alpha\ast f)(x)&=\int_{|y|\leq 2|x|}\frac{f(y)dy}{|x-y|^{N-\alpha}}+\int_{|y|> 2|x|}\frac{f(y)dy}{|x-y|^{N-\alpha}}\\[0.1in]
&\leq\int_{|y|\leq 2|x|}\frac{f(y)dy}{|x-y|^{N-\alpha}}+2^{N-\alpha}\int_{|y|> 2|x|}\frac{f(y)dy}{|y|^{N-\alpha}}\\
&<\infty.
\end{aligned}
$$
\end{proof}

In the same spirit to the above proof, if $f\ge 0$ and \eqref{Rwell} (or, equivalently \eqref{Rup}) holds, then
\begin{equation}\label{Rlow}
	I_\alpha*f(x)
	\ge 	\frac{c}{|x|^{N - \alpha}}\int_{B_{|x|}(0)} f(y) \,dy.\qedhere
\end{equation}

One of the elementary yet important for our approach consequences of \eqref{Rlow} is the following estimate, which we will be using frequently, and which to some extent is the counterpart of the Harnack inequalities on the annuli in the study of \eqref{eqm}.

\begin{lemma}\label{eq-quant}
	Let $\alpha\in(0,N)$, $\theta>0$ and  $0\le f\in L^1\big((1+|x|)^{-(N-\alpha)}dx,\R^N\big)$. Then for all $R>0$ we have
	\begin{equation}\label{eRieszint}
		\int_{B_{2R}\setminus B_R} \big(I_\alpha*f\big)^\theta\ge CR^{N-(N-\alpha)\theta}\Big(\int_{B_{R}}f\Big)^\theta.
	\end{equation}
\end{lemma}

\begin{proof}
	Follows from \eqref{Rlow} by integration.
\end{proof}

An obvious implication of \eqref{Rlow} is that $I_\alpha*f$ can not decay faster than $I_\alpha$ at infinity, even if the function $f$ is compactly supported.
Recall also that if $f\ge 0$ then
an elementary estimate shows that for every $x\in\R^N$,
\begin{equation}\
	I_\alpha*f(x)\ge\frac{A_\alpha}{R^{N-\alpha}}\int_{B_R(x)}f(y)dy.
\end{equation}
As a consequence, if $u\ge 0$ is a solution of \eqref{eqC} that is positive on a set of positive measure, then $u$ is everywhere strictly positive on $\R^N$ and the following lower bounds must  hold:
\begin{align}\label{2low}
u(x)&\ge c(1+|x|)^{-(N-\alpha)},\\
I_\beta \ast u^p(x)&\ge c(1+|x|)^{-(N-\beta)}.\label{betalow}
\end{align}
On the other hand, \eqref{eqCdom} requires
\begin{align}
u^p&\in L^1\big((1+|x|)^{-(N-\beta)}dx,\R^N\big),\label{eup}\\
(I_\beta \ast u^p)u^q&\in L^1\big((1+|x|)^{-(N-\alpha)}dx,\R^N\big).\label{eIbeta}
\end{align}
Combining the competing upper and lower bounds immediately leads to the following nonexistence result.

 \begin{proposition}\label{p1}
 	Let $p,q>0$ and assume that either $p\le\frac{\beta}{N-\alpha}$, or $\alpha+\beta>N$ and $q\le\frac{\beta}{N-\alpha}-1$.
 	If $u \ge 0$ is a solution of \eqref{eqC} then $u \equiv 0$.
 \end{proposition}

 \begin{proof}
First we note that \eqref{2low} and \eqref{eup} are incompatible when $p\le\frac{\beta}{N-\alpha}$. Then we observe that \eqref{2low}, \eqref{betalow} and \eqref{eIbeta} are incompatible when $0<q\le\frac{\beta}{N-\alpha}-1$.
\end{proof}

\begin{remark}
We will see in Proposition \ref{p47} below that $q\le\frac{\beta}{N-\alpha}-1$ is suboptimal for the nonexistence and could be refined.
\end{remark}

\section{Nonlocal positivity principle}

The nonexistence result in Proposition \ref{p1} "decouples'' the values of $p$ and $q$. In order to deduce an estimate which involves the quantity $p+q$ which appears in Theorem \ref{t1}, we need the following lemma, inspired by \cite[Proposition 2.1]{MVS-H} and \cite{MVS-JDE}*{Section 3}.

	\begin{lemma}[Nonlocal positivity principle]
		\label{lV}
		Let $\alpha\in(0,N)$ and $V:\R^N\to[0,\infty)$ be a measurable function.
		Assume that there exists $u\in L^1_{loc}(\R^N)$ such that $u>0$ a.e.~in $\R^N$, $Vu\in L^1\big((1+|x|)^{-(N-\alpha)}dx,\R^N\big)$ and
		\begin{equation}\label{eqV}		
			u \ge I_{\alpha}\ast (Vu)\quad\text{a.e.~in}\quad\R^N.
		\end{equation}
		Then for every $R>0$ and $0\le \varphi\in C^\infty_c(B_R)$,
		\begin{equation}\label{nonlocal-V}
		\int_{B_R}\varphi^2\ge CR^{\alpha-N}\left(\int_{B_R} \sqrt{V}\varphi\right)^2.
		\end{equation}
	\end{lemma}

\begin{proof}
Take $\psi:=\frac{\varphi^2}{u}$ as a test function in \eqref{eqV}. Then
\begin{align*}
\int_{B_R}\varphi^2&\ge\iint_{B_R\times B_R} I_\alpha(|x-y|)V(y)u(y)\frac{\varphi^2(x)}{u(x)}dy\,dx\\
&=\frac12\iint_{B_R\times B_R} I_\alpha(|x-y|)\left(V(x)u(x)\frac{\varphi^2(y)}{u(y)}+V(y)u(y)\frac{\varphi^2(x)}{u(x)}\right)dxdy\\
&=\iint_{B_R\times B_R} I_\alpha(|x-y|)\sqrt{V(x)V(y)}\varphi(y)\varphi(x)dxdy\\
&+\frac12\iint_{B_R\times B_R} I_\alpha(|x-y|)u(x)u(y)\left(\sqrt{V(x)}\frac{\varphi(y)}{u(y)}-\sqrt{V(y)}\frac{\varphi(x)}{u(x)}\right)^2dxdy\\
&\ge \frac{A_\alpha}{2^{N - \alpha} R^{N - \alpha}}\left(\int_{B_R}\sqrt{V(x)}\varphi(x)dx\right)^2,
\end{align*}
since
\begin{equation}\label{eqRieszlow}
I_\alpha(x-y) \ge \frac{A_\alpha}{2^{N - \alpha} R^{N - \alpha}}\qquad(x, y \in B_R),
\end{equation}
which completes the proof.
\end{proof}

\begin{remark}
{\rm Nonlocal inequality \eqref{eqV} can be interpreted as the ``inversion'' of the fractional Schr\"odinger inequality
$$(-\Delta)^{\alpha/2}u-Vu\ge 0\quad\text{in}\quad\R^N.$$
In this context Lemma \ref{lV} can be seen as a higher--order version of the fractional
Agmon--Allegretto--Piepenbrink's positivity principle: if \eqref{eqV} has a positive supersolution then a certain variational inequality which involves the potential $V$ must hold.
We will see that Lemma \ref{lV} alongside with the standard integral estimate \eqref{eRieszint} of the Riesz potentials are sufficient for the complete analysis of the existence and nonexistence of positive solutions of the nonlinear inequality \eqref{eqC}.}
\end{remark}

Using Lemma \ref{lV} we establish the following estimate.
	
	\begin{proposition}
		\label{ppq}
		Let $p,q>0$ and $u>0$ be a solution of \eqref{eqC}.
		Then, for every $R>0$ and every $\varphi \in C^\infty_c(B_R)$,
		\begin{equation}\label{nonlocal}
		\int_{B_R}\varphi^2\ge C R^{\alpha+\beta-2N}\left(\int_{B_R}u^p\right)\left(\int_{B_R} u^\frac{q-1}{2}\varphi\right)^2.
		\end{equation}
	\end{proposition}

	\begin{proof}
		For every $\varphi\in C^\infty_c(B_R)$, by Lemma \ref{lV} with  $V=(I_\beta \ast u^p)u^{q-1}$, and using a similar inequality to \eqref{eqRieszlow} for $I_\beta \ast u^p$, we obtain
		\begin{align*}\label{nonlocal-V}
		\int_{B_R}\varphi^2&\ge CR^{\alpha-N}\left(\int_{B_R} \big((I_\beta \ast u^p)u^{q-1}\big)^{1/2}\varphi\right)^2\\
		&\ge C'R^{\alpha+\beta-2N}\left(\int_{B_R}u^p\right)\left(\int_{B_R} u^\frac{q-1}{2}\varphi\right)^2.\qedhere
		\end{align*}
\end{proof}

One of the principal tools in the subsequent analysis is the following decay estimate on the solutions of \eqref{eqC}, which is an adaptation of \eqref{nonlocal}. Note that for $q<1$ our estimate contains a lower bound on the solution, since the 2nd integral involves a negative power of $u$.

\begin{corollary}\label{c1}
Let $p,q>0$ and $u>0$ be a solution of \eqref{eqC}.
Then for every $R>0$,
\begin{equation}\label{nonlocal4}
\left(\int_{B_{2R}}u^p\right)\left(\int_{B_{2R}\setminus B_R} u^\frac{q-1}{2}\right)^2 \leq CR^{3N-\alpha-\beta}.
\end{equation}
\end{corollary}

\begin{proof}
Take $\varphi_R(x)=\varphi(x/R)$, where $\varphi \in C^\infty_c(\R^N)$ is such that $\mathrm{supp}(\varphi) \subset \overline B_{4} \setminus  B_{1/2}$, $\varphi \equiv 1$ on $\overline B_{2}\setminus  B_{1}$ and $0\le \varphi \le 1$. Then,by \eqref{nonlocal} we find
\begin{align*}\label{nonlocal2}
cR^N=\int_{B_{4R}}\varphi_R^2&\ge C' R^{\alpha+\beta-2N}\left(\int_{B_{4R}}u^p\right)\left(\int_{B_{4R}} u^\frac{q-1}{2}\varphi_R\right)^2\\
&\ge  C' R^{\alpha+\beta-2N}\left(\int_{B_{2R}}u^p\right)\left(\int_{B_{2R}\setminus B_R} u^\frac{q-1}{2}\right)^2.\qedhere
\end{align*}
\end{proof}

\section{Nonexistence}\label{none}

In this section we derive several nonexistence result for \eqref{eqC}. Our approach is inspired by
\cite{MVS-JDE} which studied the inequality \eqref{eqm} in the semilinear the case $\alpha=2$, yet with substantial modifications. In particular, in this work we completely avoid the use of the comparison principle and Harnack's inequalities, which are not applicable in the framework of the double--nonlocal inequality \eqref{eqC}. It turns out that Harnack inequality estimates in the context of \eqref{eqC} can be replaced by the estimate \eqref{eRieszint}.

\begin{proposition}\label{p5}
	Let $p,q>0$ and assume that \(p+q<1\).
	If \(u \ge 0\) is a solution of \eqref{eqC} then \(u \equiv 0\).
\end{proposition}

\begin{proof}
		Since  $u^p\in L^1((1+|x|)^{-(N-\beta)}dx)$ and $\beta<N$, by Lebesgue's dominated convergence theorem
	\[
	\int_{B_{2R} \setminus B_R} u^p = o\big(R^{N-\beta}\big)=o\big(R^{N}\big) \quad\mbox{ as }R\to \infty.
	\]
 Since \(p+q  < 1\), we may apply H\"older and then the estimate  \eqref{nonlocal4}  to obtain
	\begin{align*}
	cR^N=\int_{B_{2R} \setminus B_R} 1 &= \int_{B_{2R} \setminus B_R} u^\frac{p(1-q)}{2p+1-q} \; u^\frac{(q-1)p}{2p+1-q}\\
	&\le \Bigl( \int_{B_{2R} \setminus B_R} u^p \Bigr)^\frac{1 - q}{2p + 1 - q}\Bigl( \int_{B_{2R} \setminus B_R} u^\frac{q - 1}{2} \Bigr)^\frac{2p}{2p + 1 - q}\\
	&\le \Bigl( \int_{B_{2R} \setminus B_R} u^p \Bigr)^\frac{1 - q - p}{2p + 1 - q}\biggl[\Bigl( \int_{B_{2R} \setminus B_R} u^p \Bigr) \Bigl( \int_{B_{2R} \setminus B_R} u^\frac{q - 1}{2} \Bigr)^2\biggr]^\frac{p}{2p + 1 - q}\\
	&\le \Bigl( \int_{B_{2R} \setminus B_R} u^p \Bigr)^\frac{1 - q - p}{2p + 1 - q}\Bigl( CR^{3N-\alpha-\beta}\Bigr)^\frac{p}{2p + 1 - q}\\
	&\le o(1)R^{N\frac{1 - q - p}{2p + 1 - q}}\bigl(R^{3N}\bigr)^\frac{p}{2p + 1 - q}\\
	&=o(1)R^N,
	\end{align*}
 which raises a contradiction.
\end{proof}

\begin{proposition}\label{p6}
	Let $p,q>0$ and assume that $1\le p+q \le \frac{N+\beta}{N - \alpha}$.
	If \(u \ge 0\) is a solution of \eqref{eqC} then \(u \equiv 0\).
\end{proposition}

\begin{proof} Assume first \(p + q<\frac{N + \beta}{N - \alpha}\).
	By the Cauchy--Schwarz inequality,
	$$\int_{B_{2R}} u^p\ge cR^{-N}\left(\int_{B_{2R}}u^\frac{p}{2}\right)^2,$$
	and so \eqref{nonlocal4} implies
		\begin{equation}\label{nonlocal4-2}
		CR^{4N-\alpha-\beta}\ge\left(\int_{B_{2R}}u^\frac{p}{2}\right)^2\left(\int_{B_{2R}\setminus B_R} u^\frac{q-1}{2}\right)^2.
	\end{equation}
	Using \eqref{nonlocal4-2} and the Cauchy--Schwarz again together with $u\ge c|x|^{-(N-\alpha)}$ in $\R^N\setminus B_1$ (that follows from \eqref{2low}), we obtain
	\begin{equation}\label{eq19}
\begin{aligned}	
		CR^{4N-\alpha-\beta} & \ge \Bigl(\int_{B_{2R}} u^\frac{p}{2} \Bigr)^2 \Bigl(\int_{B_{2R} \setminus B_R} u^{\frac{q - 1}{2}}\Bigr)^2\\[0.1in]
		& \ge \Big(\int_{B_{2R} \setminus B_R} u^\frac{p+q-1}{4} \Big)^4\\[0.1in]
		&\ge cR^{4N-(N - \alpha)(p+q-1)},
		\end{aligned}
	\end{equation}
which is a contradiction since $0<p+q-1<\frac{\alpha+\beta}{N-\alpha}$.
	
Assume now \(p + q=\frac{N + \beta}{N - \alpha}\). By H\"older's inequality we find
$$
\begin{aligned}
\Big(\int_{\R^N} (I_\beta \ast u^p) u^q\Big)^2&=\Big(\iint_{\R^N\times \R^N}I_\beta(x-y)u^p(y)u^q(x)\Big)
\Big(\iint_{\R^N\times \R^N}I_\beta(x-y)u^p(x)u^q(y)\Big)\\
&\geq \Big(\iint_{\R^N\times \R^N} u (x)^{\frac{p + q}{2}} I_\beta (x - y) u (y)^{\frac{p + q}{2}}\,dx\,dy\Big)^2,\\
\end{aligned}
$$
so that
$$
\int_{\R^N} (I_\beta \ast u^p) u^q\geq \iint_{\R^N\times \R^N} u (x)^{\frac{p + q}{2}} I_\beta (x - y) u (y)^{\frac{p + q}{2}}.
$$
Using the lower bound \eqref{2low} and the fact that \(\frac{p + q}{2} = \frac{N + \beta}{2(N - \alpha)} > 0\) we deduce
	\[
	\begin{split}
		\int_{\R^N} (I_\beta \ast u^p) u^q & \ge
		\int_{\R^N \setminus B_{1}} \int_{\R^N \setminus B_{1}} u (x)^{\frac{p + q}{2}} I_\beta (x - y) u (y)^{\frac{p + q}{2}}\,dx\,dy\\
		&\ge c\int_{\R^N \setminus B_{1}} \int_{\R^N \setminus B_{1}} \frac{1}{|x|^\frac{N + \beta}{2}} I_\beta(x - y) \frac{1}{|y|^\frac{N + \beta}{2}}\,dx\,dy = \infty.
	\end{split}
	\]
Hence,
\begin{equation}\label{eea1}
\lim_{R\to \infty} \int_{B_R} (I_\beta \ast u^p) u^q=\infty.
\end{equation}
Since $u$ satisfies \eqref{eqC}, from Lemma \ref{eq-quant} with $\theta=\frac{p + q - 1}{4}=\frac{\alpha+\beta}{4(N-\alpha)}$ and $f=(I_\beta\ast u^p)u^q$ we find
$$
\int_{B_{2R} \setminus B_R} u^\frac{p + q - 1}{4}\geq \int_{B_{2R} \setminus B_R} (I_\alpha\ast f)^\frac{p + q - 1}{4}\geq CR^{N-\frac{\alpha+\beta}{4}}\Big(\int_{B_R} (I_\beta\ast u^p)u^q \Big)^{\frac{p+q-1}{4}}.
$$
From the above estimate and \eqref{eea1} we deduce
$$
\lim_{R\to \infty}  \frac{1}{R^{N-\frac{\alpha+\beta}{4}}} \int_{B_{2R} \setminus B_R} u^\frac{p + q - 1}{4}=\infty,
$$
which contradicts  the upper bound in \eqref{eq19}.
\end{proof}

If \(\alpha+\beta \ge N\) we give precise lower bounds on \(\int_{B_{2 R} \setminus B_R} u^{q - 1}\) to obtain a further nonexistence result.

\begin{proposition}\label{p47}
	Let $p,q>0$ and assume that $\alpha+\beta>N$ and $1<q\le\frac{\beta}{N-\alpha}$.
	If \(u \ge 0\) is a solution of \eqref{eqC} then \(u \equiv 0\).
\end{proposition}

\begin{proof}
	Assume that $u > 0$ on a set of positive measure.
	From \eqref{2low} we obtain
	\begin{equation}
		\label{ineqLoweruq_1}
		\left(\int_{B_{2R} \setminus B_R} u^\frac{q - 1}{2}\right)^2 \ge c R^{2N - (N - \alpha)(q - 1)}.
	\end{equation}
	On the other hand, by Corollary \ref{c1},
	\begin{equation}
		\label{ineqLoweruq_2}
	\left(\int_{B_{2R} \setminus B_R} u^\frac{q - 1}{2}\right)^2 \le \frac{C R^{3 N- \beta - \alpha}}{\displaystyle \int_{B_{2R}} u^p} \le C' R^{3 N- \beta - \alpha}.
\end{equation}
	This yields a contradiction if $q < \frac{\beta}{N - \alpha}$.
	
	In the critical case $q = \frac{\beta}{N - \alpha}$ using \eqref{Rlow}, \eqref{betalow} we obtain
	\begin{align*}
		\int_{B_R}  (I_\beta \ast u^p) u^q & \ge c\left( \int_{B_1} u^p \right) \int_{B_R\setminus B_1} \frac{u (x)^q}{|x|^{N - \beta}}dx \\
		&\ge c' \int_{B_R \setminus B_1} \frac{1}{|x|^{N - \beta + (N - \alpha) q}} \,dx =c''\log(R),
	\end{align*}
	since $(N - \alpha) q = \beta$.
	By Lemma \ref{eq-quant} with $\theta=\frac{q-1}{2}>0$,
\begin{equation}\label{eqlog1}	
\begin{aligned}
\int_{B_{2R}\setminus B_R}u^\frac{q-1}{2} & \ge\int_{B_{2R}\setminus B_R} \Big(I_{\alpha}* \big((I_\beta \ast u^p) u^q\big)\Big)^\frac{q-1}{2}\\
& \ge cR^{N-(N-\alpha)\frac{q-1}{2}}\left( \int_{B_{R}}(I_\beta \ast u^p) u^q\right)^{\frac{q-1}{2}}\\
& \ge c'R^{\frac{2N-(N-\alpha)(q-1)}{2}}\log^\frac{q-1}{2}(R)\\
& =c'R^{\frac{3N-\alpha-\beta}{2}}\log^\frac{q-1}{2}(R),
\end{aligned}
\end{equation}
which contradicts \eqref{ineqLoweruq_2}.
\end{proof}

The transitional locally linear case \(q=1\) requires a separate consideration.

\begin{proposition}\label{p8}
	Let $p>0$ and assume that \(\alpha+\beta > N\) and \(q=1\).
	If \(u \ge 0\) is a solution of \eqref{eqC} then \(u \equiv 0\).
\end{proposition}

\begin{proof} 
	Using Corollary \ref{c1}, for any \(R>0\) we have
		\[
		R^{3N - \alpha-\beta}\ge c\Bigl(\int_{B_R} u^{p}\Bigr)\Bigl(\int_{B_{2R} \setminus B_R} 1\Bigr)^2=cR^{2N}\Bigl(\int_{B_R} u^{p}\Bigr).
		\]
		Since $\alpha+\beta>N$, it follows that \(u \equiv 0\). 
\end{proof}

In the sublinear case $q<1$ we deduce an additional restriction on the admissible range of the exponent $q$.

\begin{proposition}\label{p9}
	Let $p,q>0$ and assume that \(p+q\ge 1\), \(q < 1\) and
	\[
	q \le 1 - \frac{N - \alpha- \beta}{N}p.
	\]
	If \(u \ge 0\) is a solution of \eqref{eqC} then \(u \equiv 0\).
\end{proposition}
\begin{proof}
		Since $q<1$,  by H\"older's inequality we deduce		
		\begin{equation}\label{ee1}
		\begin{aligned}
			cR^N=\int_{B_{2R} \setminus B_R} 1 &= \int_{B_{2R} \setminus B_R} u^{p\frac{1-q}{2p+1-q}} \, u^{\frac{q-1}{2}\frac{2p}{2p+1-q}}\\
			&\le \Bigl( \int_{B_{2R} \setminus B_R} u^p \Bigr)^\frac{1 - q}{2p + 1 - q}\Bigl( \int_{B_{2R} \setminus B_R} u^\frac{q - 1}{2} \Bigr)^\frac{2p}{2p + 1 - q}\\
			&=\biggl[\Bigl( \int_{B_{2R} \setminus B_R} u^p \Bigr)\Bigl( \int_{B_{2R} \setminus B_R} u^\frac{q - 1}{2} \Bigr)^2\biggr]^\frac{1 - q}{2p + 1 - q}\biggl(\int_{B_{2R} \setminus B_R} u^\frac{q - 1}{2} \biggr)^{2\frac{p + q - 1}{2p + 1 - q}}.
		\end{aligned}
		\end{equation}
		By Corollary \ref{c1} we have
$$
\left(\int_{B_{2R}}u^p\right)\left(\int_{B_{2R}\setminus B_R} u^\frac{q-1}{2}\right)^2 \leq CR^{3N-\alpha-\beta},
$$
which yields
\begin{equation}\label{ee2}
		\Bigl(\int_{B_{2R} \setminus B_R} u^p \Bigr)\Bigl( \int_{B_{2R} \setminus B_R} u^\frac{q - 1}{2} \Bigr)^2 \le C R^{3N - \alpha-\beta},
\end{equation}
		and on the other hand
\begin{equation}\label{ee3}
		\Bigl( \int_{B_{2R} \setminus B_R} u^\frac{q - 1}{2}\Bigr)^2 \le C \frac{R^{3N - \beta - \alpha}}{\displaystyle \int_{B_{2R}} u^p} \le C' R^{3N -\alpha-\beta}.
\end{equation}

If $q < 1 - \frac{N - \beta - \alpha}{N}p$, we use \eqref{ee2}-\eqref{ee3} in  \eqref{ee1} to raise a contradiction since \(p + q \ge 1\).
	\smallskip
	
If \(q = 1 - \frac{N - \beta - \alpha}{N}p\), we use \eqref{ee3} and H\"older's inequality to further estimate
	\[
	\int_{B_{2R} \setminus B_R} u^p \ge \frac{\Bigl(\displaystyle{\int_{B_{2R} \setminus B_R} 1}\Bigr)^{1 + \frac{2p}{1 - q}}} {\displaystyle{\Bigl(\int_{B_{2R} \setminus B_R} u^\frac{1 - q}{2} }\Bigr)^\frac{2p}{1 - q} } \ge c R^{N-\frac{p}{1 - q} (N -\alpha- \beta)}= c.
	\]
	This shows that $u^p\not\in L^1(\R^N)$ and
	\[
	\lim_{R \to \infty} \int_{B_{2R}} u^p = \infty.
	\]
Using this fact in \eqref{ee3} we deduce
	\[
	\Big(\int_{B_{2R} \setminus B_R} u^\frac{q - 1}{2}\Big)^2 = o(R^{3 N-\alpha-\beta}) \quad\mbox{ as }R\to \infty.
	\]
	We now use this last estimate and \eqref{ee2} in \eqref{ee1} to conclude.
\end{proof}

\section{Optimal decay and existence}\label{odec}

If \(u\ge 0\) is a solution of \eqref{eqC}, then either \(u \equiv 0\) or
$u$ must obey the "natural'' lower bound \eqref{2low}, which implies in particular, the integral lower bound
\begin{equation}\label{2lowint}
\int_{B_{2R} \setminus B_R}u\ge cR^{\alpha}.
\end{equation}
In the region $q<1$ the estimate \eqref{nonlocal4} of Corollary \ref{c1} leads to an integral lower bound which improves upon \eqref{2lowint} when  $\alpha+\beta<N$ and $q < \frac{\beta}{N - \alpha}$.

\begin{proposition}
	\label{p10}
	Let $p,q>0$ and assume that $\alpha+\beta<N$ and $q < \frac{\beta}{N - \alpha}{<1}$.
	If \(u\ge 0\) is a solution of \eqref{eqC}, then either \(u \equiv 0\) or
	\begin{equation}\label{qlowint}
		\int_{B_{2R} \setminus B_R}u\ge cR^{\frac{\alpha+\beta-Nq}{1-q}}.
	\end{equation}
\end{proposition}
As pointed out in Remark~\ref{remark01}, since $\alpha+\beta<N$ and $q < \frac{\beta}{N - \alpha}<1$, the exponent of $R$ in \eqref{qlowint} is greater than $\alpha$.

\begin{proof}
	From Corollary \ref{c1} we have
	\begin{equation*}
	\Big(\int_{B_{2R} \setminus B_R} u^{-\frac{1-q}{2}}\Big)^2 \le c\frac{R^{3N - \alpha-\beta}}{\displaystyle \int_{B_{2R}} u^{p}} \le c' R^{3N - \alpha-\beta}.
\end{equation*}
Further, by H\"older's inequality (since $0<q<1$) we have 	
	\begin{equation*}
		\Big(\int_{B_{2R} \setminus B_R} u^{-\frac{1-q}{2}}\Big)^2\ge \Big(\int_{B_{2R} \setminus B_R} u\Big)^{-(1-q)}\Big(\int_{B_{2R} \setminus B_R} 1\Big)^{3-q}.
	\end{equation*}
Now, the above estimates yield
\begin{equation*}
	\Big(\int_{B_{2R} \setminus B_R} u\Big)^{-(1-q)}\le R^{Nq-\alpha-\beta},
\end{equation*}
which leads to \eqref{qlowint}.
\end{proof}

Our next step is to construct explicit solutions with the decay which match or near-match the lower bounds in \eqref{2lowint} and \eqref{qlowint}. Before we do this, we recall the following simple estimates, cf. \cite[Lemma A.1 and A.2]{MVS-JDE} which are frequently used in the proofs below.

\begin{lemma}\label{A1}
	Let \(v\in L^1_{loc}(\R^N)\), \(\gamma \in (0, N)\) and \(s > \gamma\).
	If
	\[
	\limsup_{|x| \to \infty} v(x) |x|^s <\infty,
	\]
	then
	\begin{align*}
		\limsup_{|x| \to \infty} (I_\gamma \ast v)(x)|x|^{s-\gamma} &<\infty& & \text{if \(\gamma<s<N\)},\\
		\limsup_{|x| \to \infty} (I_\gamma \ast v)(x) \frac{|x|^{N - \gamma}}{\log |x|} &<\infty&
		&\text{if \(s=N\)},\bigskip\\
		\limsup_{|x| \to \infty} (I_\gamma \ast v)(x)|x|^{N - \gamma} &<\infty& & \text{if \(s>N\)}.
	\end{align*}
\end{lemma}

\begin{lemma}\label{A2}
	Let \(v\in L^1_{loc}(\R^N)\), \(\gamma \in (0, N)\) and \(\sigma\in\R\).
	If
	\[
	\limsup_{|x| \to \infty} v(x) \frac{|x|^N}{(\log |x|)^\sigma} <\infty,
	\]
	then
	\begin{align*}
		\limsup_{|x| \to \infty} (I_\gamma \ast v)(x) |x|^{N - \gamma} &<\infty & &\text{if \(\sigma < -1\)},\\
		\limsup_{|x| \to \infty} (I_\gamma \ast v)(x) \frac{|x|^{N - \gamma}}{(\log (\log |x|))} &<\infty & &\text{if \(\sigma = -1\)},\\
		\limsup_{|x| \to \infty} (I_\gamma \ast v)(x) \frac{|x|^{N - \gamma}}{(\log |x|)^{\sigma+1}} &<\infty & &\text{if \(\sigma > -1\)}.
	\end{align*}
\end{lemma}

\begin{proposition}\label{sharp-B1}
	Assume that
	\begin{equation}\label{cdd1}
		p > \frac{\beta}{N - \alpha},\quad
		p + q > \frac{N + \beta}{N - \alpha}\quad
		\text{and}\quad
		q > \frac{\beta}{N - \alpha}.
	\end{equation}
	Then, \eqref{eqC} admits a positive radial solution $u \in C(\R^N)$ which satisfies
	\begin{equation}\label{ulim}
	\limsup_{|x| \to \infty} u (x) |x|^{N - \alpha} <\infty.
	\end{equation}
\end{proposition}

\begin{proof} Let $0<\varepsilon<q(N-\alpha)-\beta$ and take $u(x)=(1+|x|)^{-(N-\alpha)}$. Since $p(N-\alpha)>\beta$ we can apply the estimates in Lemma \ref{A1} to deduce
$$
(I_\beta\ast u^p)(x)\leq  c_1
\begin{cases}
(1+|x|)^{\beta-p(N-\alpha)}&\quad\mbox{ if }p(N-\alpha)<N\\
(1+|x|)^{\beta-N}&\quad\mbox{ if }p(N-\alpha)>N\\
(1+|x|)^{\beta-N}\log  (|x|+e) &\quad\mbox{ if }p(N-\alpha)=N\\
\end{cases}
\quad\mbox{ in }\R^N,
$$
for some constant $c_1>0$. Thus,
$$
\Big[(I_\beta\ast u^p)u^q\Big](x)\leq  c_1
\begin{cases}
(1+|x|)^{\beta-(p+q)(N-\alpha)}&\quad\mbox{ if }p(N-\alpha)<N\\
(1+|x|)^{\beta-N-q(N-\alpha)}&\quad\mbox{ if }p(N-\alpha)>N\\
(1+|x|)^{\beta-N-q(N-\alpha)}\log  (|x|+e) &\quad\mbox{ if }p(N-\alpha)=N\\
\end{cases}
\quad\mbox{ in }\R^N.
$$
In particular, one may further estimate as
$$
\Big[(I_\beta\ast u^p)u^q\Big](x)\leq  c_1
\begin{cases}
(1+|x|)^{\beta-(p+q)(N-\alpha)}&\quad\mbox{ if }p(N-\alpha)<N\\
(1+|x|)^{\beta-N-q(N-\alpha)+\varepsilon}&\quad\mbox{ if }p(N-\alpha)\geq N\\
\end{cases}
\quad\mbox{ in }\R^N.
$$
Since $(p+q)(N-\alpha)-\beta>N$ and $N-\beta+q(N-\alpha)-\varepsilon>N$ it follows from the third estimate in Lemma \ref{A1} that
$$
I_{\alpha}\ast \Big[(I_\beta\ast u^p)u^q\Big](x)\leq  c_2
(1+|x|)^{\alpha-N} =c_2u(x) \quad\mbox{ in }\R^N,
$$
where $c_2>0$ is a constant.
Thus, the continuous function $U(x)=c_2^{-1/(p+q-1)}u(x)$ is a solution of \eqref{eqC} which satisfies \eqref{ulim}.
\end{proof}

\begin{proposition}\label{sharp-B2}
	Assume that
	\begin{equation*}
		1-\frac{N - \alpha - \beta }{N} p <  q < \frac{\beta}{N - \alpha} < 1.
	\end{equation*}
Then, \eqref{eqC} admits a positive radial solution \(u \in C(\R^N)\)
	which satisfies
	\[
	\limsup_{|x| \to \infty} u (x) |x|^{\frac{N - \alpha-\beta}{1-q}} < \infty.
	\]
\end{proposition}
\begin{proof}
Let $u(x)=(1+|x|)^{-k}$ where $k=(N-\alpha-\beta)/(1-q)$.
Since $1-(N-\alpha-\beta)p/N<q$, we have $pk>N$, and hence by the third estimate of Lemma \ref{A1} we have
$$
I_{\beta}*u^p\le c_1(1+|x|)^{\beta-N} \quad\mbox{ in }\R^N,
$$
for some constant $c_1>0$.
Since $\beta-N-kq=-\frac{N-\beta-\alpha q}{1-q}$, we have
$$
(I_{\beta}*u^p)u^q\le c_2(1+|x|)^{-\frac{N-\beta-\alpha q}{1-q}} \quad\mbox{ in }\R^N.
$$
Since $q<\frac{\beta}{N-\alpha}<1$, we have $\alpha<\frac{N-\beta-\alpha q}{1-q}<N$.
Hence, by the first estimate of Lemma \ref{A1} we have
$$
 I_{\alpha}*\Big[ (I_{\beta}*u^p)u^q\Big]\le c_2(1+|x|)^{-k} \quad\mbox{ in }\R^N,
$$
where $c_2>0$ is a constant. Thus, $U(x)=c_2^{-1/(p+q-1)}(1+|x|)^{-k}$ is a continuous solution of \eqref{eqC}.
Moreover,
$$
\limsup_{|x|\to\infty}U(x)|x|^k<\infty.
$$
\end{proof}

\begin{proposition}\label{sharp-B1plus}
	Assume that
	$$\alpha+\beta=N,\quad p>\frac{N}{N-\alpha}\quad\text{and}\quad q=1.$$
	Then, for every \(m > 0\) inequality \eqref{eqC} admits a positive radial solution
	\(u \in C(\R^N)\)	which satisfies
	\[
	\limsup_{|x| \to \infty} u (x) |x|^{N - \alpha-m} < \infty.
	\]
\end{proposition}

\begin{proof}
Let $m>0$.
Since $p>\frac{N}{N-\alpha}=\frac{N}{\beta}$, we see that $\beta-\frac{N}{p}>0$. Set
$$
\delta=
\begin{cases}
m & \mbox{ if }\, 0<m<\beta-\frac{N}{p},\\
\frac{1}{2}(\beta-\frac{N}{p}) & \mbox{ if }\, m\ge \beta-\frac{N}{p},
\end{cases}
\quad\mbox{ and }\quad k=\beta-\delta.
$$
Then,
$$
kp=
\begin{cases}
(\beta-m)p>N & \mbox{ if }\, 0<m<\beta-\frac{N}{p},\\
\frac{\beta p}{2}+\frac{N}{2}>N & \mbox{ if }\, m\ge \beta-\frac{N}{p}.
\end{cases}
$$
Let $u(x)=(1+|x|)^{-k}$.
By the third estimate of Lemma \ref{A1} we see that $I_{\beta}*u^p\le c_1(1+|x|)^{\beta-N}$ in $\R^N$ for some constant $c_1>0$, and hence
$$
(I_{\beta}*u^p)u^q\le c_1(1+|x|)^{\beta-N-kq} \quad\mbox{ in }\R^N.
$$
Since  $\alpha<-\beta+N+kq<N$, by the first estimate of Lemma \ref{A1} we see that
$$
I_{\alpha}*\Big[ (I_{\beta}*u^p)u^q\Big] \le c_2(1+|x|)^{\alpha+\beta-N-kq} \quad\mbox{ in }\R^N,
$$
for some constant $c_2>0$.
Since $\alpha+\beta-N-kq=-k$, we have
$$
I_{\alpha}*\Big[(I_{\beta}*u^p)u^q\Big]\le c_2u \quad\mbox{ in }\R^N.
$$
Thus, $U(x)=c_2^{-1/(p+q-1)}(1+|x|)^{-k}$ is a continuous solution of \eqref{eqC}.
Moreover,
$$
\limsup_{|x|\to\infty}U(x)|x|^{N-\alpha-m}\le\limsup_{|x|\to\infty}U(x)|x|^k<\infty.
$$

\end{proof}

\begin{proposition}\label{sharp-B1plusplus}
	Assume that
	\[
	p>\frac{N}{N-\alpha}\quad\text{and}\quad q = \frac{\beta}{N - \alpha}<1.
	\]
	Then, for $m\ge\frac{N - \alpha}{N -\alpha - \beta}$ inequality \eqref{eqC} admits a positive radial solution
	\(u \in C(\R^N)\),
	which satisfies
	\[
	\limsup_{|x| \to \infty} u (x) |x|^{N - \alpha} \bigl(\log |x|\bigr)^{-m} < \infty.
	\]
\end{proposition}

\begin{proof}
	Take $$u(x)=(1+|x|)^{-(N-\alpha)}(\log(e+|x|))^m.$$
	Since $-(N-\alpha)p<-N$, by the third estimate of Lemma \ref{A1} we see that $I_{\beta}*u^p\le c_1(1+|x|)^{\beta-N}$ for some constant $c_1>0$, and hence
	$$
	(I_{\beta}*u^p)u^q\le c_1(1+|x|)^{\beta-N-(N-\alpha)q}(\log(e+|x|))^{mq}=c_1(1+|x|)^{-N}(\log(e+|x|))^{mq}.
	$$
 By the third estimate of Lemma \ref{A2} we see that
	$$
	I_{\alpha}*\Big[ (I_{\beta}*u^p)u^q\Big] \le c_2(1+|x|)^{-(N-\alpha)}(\log(e+|x|))^{mq+1}\quad\mbox{ in }\R^N,
	$$
	for some constant $c_2>0$.
	Since $m\ge \frac{N - \alpha}{N -\alpha - \beta}$, we have
	$$
	I_{\alpha}*\Big[(I_{\beta}*u^p)u^q\Big]\le c_2u.
	$$
	Thus, $U(x)=c_2^{-1/(p+q-1)}u$ is a continuous solution of \eqref{eqC}.
\end{proof}

\end{document}